\documentclass{article}%
\usepackage{fancyhdr}
\usepackage{amsmath}
\usepackage{graphicx}%
\usepackage{amsfonts}%
\usepackage{amssymb}
\newtheorem{theorem}{Theorem}

\newtheorem{example}[theorem]{Example}

\newtheorem{proposition}[theorem]{Proposition}
\newtheorem{remark}[theorem]{Remark}

\newenvironment{proof}[1][Proof]{\textbf{#1.} }{\ \rule{0.5em}{0.5em}}
\begin{document}

\title{Using vector divisions in solving linear complementarity problem}
\author{Youssef ELFOUTAYENI$^{(1)}$ and Mohamed KHALADI$^{(2)}$\\$^{1}$Computer Sciences Department\\School of Engineering and Innovation of Marrakech\\youssef.elfoutayeni@campusmarrakech.com\\$^{2}$Department of Mathematical\\Faculty Semlalia, University CADI AYYAD\\khaladi@ucam.ac.ma}
\maketitle

\begin{abstract}

\end{abstract}

The linear complementarity problem is to find vector $z$ in $\mathrm{IR}^{n}$
satisfying $z^{T}(Mz+q)=0$, $Mz+q\geqslant0,$ $z\geqslant0$, where $M$ as a
matrix and $q$ as a vector, are given data; this problem becomes in present
the subject of much important research because it arises in many areas and it
includes important fields, we cite for example the linear and nonlinear
programming, the convex quadratic programming and the variational inequalities
problems, ...

It is known that the linear complementarity problem is completely equivalent
to solving nonlinear equation $F(x)=0$ with $F$ is a function from
$\mathrm{IR}^{n}$ into itself defined by $F(x)=(M+I)x+(M-I)|x|+q$. In this
paper we propose a globally convergent hybrid algorithm for solving this
equation; this method is based on an algorithm given by Shi \cite{Y. Shi}, he
uses vector divisions with the secant method; but for using this method we
must have a function continuous with partial derivatives on an open set of
$\mathrm{IR}^{n}$; so we built a sequence of functions $\tilde{F}(p,x)\in
C^{\infty}$ which converges uniformly to the function $F(x)$; and we show that
finding the zero of the function $F$ is completely equivalent to finding the
zero of the sequence of the functions $\tilde{F}(p,x)$. We close our paper
with some numerical simulation examples to illustrate our theoretical results.

\bigskip

\bigskip


\bigskip

\begin{enumerate}
\item[\textbf{1.}] {\Large Introduction:}

The complementarity problem noted ($CP$) is a classical problem of
optimization theory of finding $(z,w)\in\mathrm{IR}^{n}\times\mathrm{IR}^{n}$
such that:%
\begin{equation}
\left\{
\begin{array}
[c]{l}%
<z,w>=0\\
w-f(z)=0\\
z,w\geqslant0
\end{array}
\right.  \label{(CP)}%
\end{equation}

where $f$, a continuous operator from $\mathrm{IR}^{n}$ into itself, is given data.

The constraint $<z,w>=0$ is called the complementarity condition since for any
$i$, $1\leqslant i\leqslant n,z_{i}=0$ if $w_{i}>0$, and vice versa; it may be
the case that $z_{i}=w_{i}=0$.

In the case that the function $f$ is a nonlinear continuous operator from
$\mathrm{IR}^{n}$ into itself, so the problem is called a \textbf{N}%
on\textbf{L}inear \textbf{C}omplementarity \textbf{P}roblem associated with
the function $f$ and noted $(NLCP)$.

In the case that the function $f$ is affine, i.e%
\[
f(z)=q+Mz,
\]

where $q$ is an element of $\mathrm{IR}^{n}$ and $M$ is a real square matrix
of order $n$, so the problem is called a \textbf{L}inear \textbf{C}%
omplementarity \textbf{P}roblem associated with the matrix $M$ and the vector
$q$ and noted $(LCP)$.

Solving the $(LCP)$ in general, however, appears to be difficult. One simple
reason is that there is no known good characterization of the nonexistence of
a solution to the system for any given function $f$.

The linear complementarity problem plays an important role in several fields
(game theory, operational research ...); moreover, Cottle et
Dantzig\cite{Cottle and Dantzig} et Lemke\cite{C. E. Lemke} have proved that
all problems of linear programming $(LP)$, convex quadratic programming
$(CQP)$, and also the problems of Nash equilibrium of a game bi-matrix can be
written as a linear complementarity problem.

The question is precisely under what conditions on the matrix $M$ and the
vector $q$ this problem admits one and only one solution, if this is the case,
how can we express this solution as a function of the matrix and vector
mentioned above. This question has not been completely resolved yet.

However, many results already exist, for instance Lemke\cite{C. E. Lemke} who
gave sufficient conditions on the matrix $M$ and the vector $q$ under which
the number of solutions of $LCP(M,q)$ is finite. Samelson \cite{H. Samelson
and R. M. Thrall and O. Wesler}, Ingeton\cite{A. W. Ingleton}, Murty\cite{K.
G. Murty1971}, Watson\cite{L. T. Watson}, Kelly \cite{L. M. Kelly and L. T.
Watson} and Cottle \cite{R. W. Cottle and J. S. Pang and R. E. Stone} have by
contrast shown that the matrix $M$ is a $P-matrix$ if and only if the linear
complementarity problem associated with a matrix $M$ and a vector $q$ has a
unique solution for all $q\in\mathrm{IR}^{n}$ (We should remind that a matrix
$M$ is called a $P-matrix$ if all principal minors are strictly positive (see
\cite{M. Fiedler}), and we\ should note that any symmetric and positive
definite matrix is a $P-matrix$, but not vice-versa).

\bigskip

\item[\textbf{2.}] {\Large Preliminaries:}

In this section, we summarize some notations which will be used in this paper.

In particular, $\mathrm{IR}^{n}$ denotes the space of real $n-$dimensional vectors,

$\mathrm{IR}_{+}^{n}:=\{x\in\mathrm{IR}^{n}:x_{i}\geqslant0,i=1..n\}$ is the
nonnegative orthant and its interior is $\mathrm{IR}_{++}^{n}:=\{x\in
\mathrm{IR}^{n}:x_{i}>0,i=1..n\}$.

With $x\in\mathrm{IR}^{n}$ we define $|x|=(|x_{1}|,..,|x_{n}|)^{T}%
\in\mathrm{IR}^{n}$.

We denote by $I$ the identity matrix.

Let $x$, $y\in\mathrm{IR}^{n}$, $x^{T}y$ or $<x,y>$ is the inner product of
the $x$ and $y$; $||x||$ is the Euclidean norm.

For $x\in\mathrm{IR}^{n}$ and $k$ a nonnegative integer, $x^{(k)}$ refers to
the vector obtained after $k$ iterations; for $1\leqslant i\leqslant n$,
$x_{i}$ refers to the $i^{th}$ element of $x$, and $x_{i}^{(k)}$ refers to the
$i^{th}$ element of the vector obtained after $k$ iterations.

Let $x$, $y\in\mathrm{IR}^{n}$, the expression $x\leqslant y$ (respectively
$x<y$) meaning that $x_{i}\leqslant y_{i}$ (respectively $x_{i}<y_{i}$) for
each $i=1..n$.

For $x\in\mathrm{IR}^{n}$ we denote by $e^{x}=(e^{x_{1}},..,e^{x_{n}})^{T}%
\in\mathrm{IR}^{n}$ and for $x\in\mathrm{IR}_{++}^{n}$ we denote by
$\ln(x)=(\ln(x_{1}),..,\ln(x_{n}))^{T}$.

The transpose of a vector is denoted by super script $T$, such as the
transpose of the vector $x$ is given by $x^{T}$.

Remember that the spectrum $\sigma(A)$ of the matrix $A$ is the set of its
eigenvalues and its spectral radius $\rho$ is given by: $\rho(A):=\sup
\{|\lambda|$ such that $\lambda\in\sigma(A)\}$.

\bigskip

\item[\textbf{3.}] {\Large Equivalent reformulation of LCP:}

It is known in \cite{K. G. Murty 1988} that the linear complementarity problem
$LCP(M,q)$ is completely equivalent to solving nonlinear equation
\[
F(x)=0
\]
with $F$ is a function from $\mathrm{IR}^{n}$ into itself defined by
\[
F(x):=(M+I)x+(M-I)|x|+q;
\]

More precisely (see \cite{K. G. Murty 1988}), on the one hand, if $x^{\ast}$
is a zero of the function $F$, then%
\begin{equation}
\left\{
\begin{array}
[c]{c}%
z^{\ast}:=|x^{\ast}|+x^{\ast}\\
w^{\ast}:=|x^{\ast}|-x^{\ast}%
\end{array}
\right.  \label{Trans}%
\end{equation}
define a solution of $LCP(M,q)$.

On the other hand, if $(z^{\ast},w^{\ast})$ is a solution of $LCP(M,q)$, then
\[
x^{\ast}:=\frac{z^{\ast}-w^{\ast}}{2}%
\]
is a zero of the function $F$.

We mention that this equation is solved by the fixed point algorithm (see
\cite{U. Schafer}), this algorithm is defined by:%
\begin{equation}
\left\{
\begin{array}
[c]{l}%
x^{(0)}\in\mathrm{IR}^{n}\text{ arbitrary,}\\
x^{(k+1)}=(I+M)^{-1}(I-M)|x^{(k)}|-(I+M)^{-1}q
\end{array}
\right.  \label{Alg U. Schafer}%
\end{equation}
For the case that $M$ is symetric and positive definite, it was shown in
\cite{W. M. G. Van Bokhoven 1981} (see also Section 9.2 in \cite{K. G. Murty
1988}) that
\[
D:=(I+M)^{-1}(I-M)
\]
it holds%
\[
||D||_{2}=\sqrt{\rho(D^{T}D)}=\sqrt{\rho(D^{2})}=\rho(D)<1,
\]
where $\rho(.)$ denotes the spectral radius of a matrix; hence (\ref{Alg U.
Schafer}) converges by the contraction-mapping theorem (see Theorem 5.1.3 in
\cite{J. M. Ortega}) and
\[
x^{\ast}=\lim\limits_{k\rightarrow+\infty}x^{(k)}%
\]
is the unique solution of the $F(x)=0$.

Therefore,
\[
w^{\ast}:=|x^{\ast}|-x^{\ast}%
\]
and%
\[
z^{\ast}:=|x^{\ast}|+x^{\ast}%
\]
define the unique solution of the $LCP(M,q)$.

We mention also that the convergence of algorithm (\ref{Alg U. Schafer}) is
only linear; in this paper, we consider the use of vector divisions with the
secant method (see \cite{Y. Shi}) in instead of the algorithm (\ref{Alg U.
Schafer}) decribed above, this algorithm has global convergence (see \cite{Y.
Shi}); but for using this algorithm we must have a function $F$ a continuous
with partial derivatives on a set of $\mathrm{IR}^{n}$; the next section can
answer this problematic.

\item[\textbf{4. }] {\Large The main result}

We consider again the function $F$ $:\mathrm{IR}^{n}\rightarrow$
$\mathrm{IR}^{n}$ defined by%
\[
F(x):=(M+I)x+(M-I)|x|+q
\]
and let's consider the sequence of functions $\tilde{F}:\mathrm{IN}^{\ast
}\mathrm{\times IR}^{n}\rightarrow$ $\mathrm{IR}^{n}$defined by%
\[
\tilde{F}(p,x):=(M+I)x+\frac{1}{p}(M-I)\ln(e^{0}+e^{px}+e^{-px})+q
\]

\begin{proposition}
: $\tilde{F}(p,x)$ converges uniformly to $F(x)$ when $p\rightarrow+\infty$.
\end{proposition}

\begin{proof}
: To show that we can start by
\begin{align*}
\frac{1}{p}\ln(e^{0}+e^{px}+e^{-px})-|x|  &  =\frac{1}{p}[\ln(e^{0}%
+e^{px}+e^{-px})+\ln(e^{-p|x|})]\\
&  =\frac{1}{p}\ln[e^{-p|x|}\ast(1+e^{px}+e^{-px})]\\
&  =\frac{1}{p}\ln(e^{-p|x|}+e^{p(x-|x|)}+e^{-p(x+|x|)})
\end{align*}
then we have%
\[
0\leqslant\frac{1}{p}\ln(e^{0}+e^{px}+e^{-px})-|x|\leqslant\frac{1}{p}\ln(3)
\]
so
\[
\frac{1}{p}\ln(e^{0}+e^{px}+e^{-px})
\]
is uniform convergence to $|x|$ when $p\rightarrow+\infty$;

Moreover, the operator $(M-I)$ is linear then we have
\[
\frac{1}{p}(M-I)\ln(e^{0}+e^{px}+e^{-px})\text{ converges uniformly to
}(M-I)|x|\text{ as }p\rightarrow+\infty
\]
and from the expression of the sequence of the functions $\tilde{F}$ and the
function $F$ we have
\[
\tilde{F}(p,x)\text{ converges uniformly to}\ F(x)\text{ when }p\rightarrow
+\infty.
\]
\end{proof}

\begin{theorem}
: Let $x^{\ast}(p)$ be a solution of the equation $\tilde{F}(p,x)=0,$ then
$x^{\ast}(p)$ is an approximation solution of $F(x)=0$ for $p$ is large enough.
\end{theorem}

\begin{proof}
:To show that, we use proposition(1) wich we can interpret as $\forall
\epsilon>0,$ $\exists$ $p^{\ast}>0$ such that for all $p>p^{\ast}$ we have
\begin{align*}
||F(x^{\ast}(p))||  &  =||F(x^{\ast}(p))-\tilde{F}(p,x^{\ast}(p))||\\
&  \leqslant\epsilon
\end{align*}
then we have for any $\epsilon>0$, $x^{\ast}(p)$ is the approximation solution
of
\[
F(x)=0.
\]
\end{proof}

\begin{remark}
: The uniqueness of the root of the function $F$ results from the uniqueness
of the solution of linear complementarity problem $LCP(M,q)$, in fact,
supposing that $x_{1}^{\ast}$ and $x_{2}^{\ast}$, two distinct roots of the
function $F$,exist, then%
\begin{align*}
z_{1}^{\ast}  &  :=|x_{1}^{\ast}|+x_{1}^{\ast}.\\
z_{2}^{\ast}  &  :=|x_{2}^{\ast}|+x_{2}^{\ast}.
\end{align*}
Since $z_{1}^{\ast}=z_{2}^{\ast}$ (uniqueness of the solution of $LCP(M,q)$)
then
\begin{equation}
|x_{1}^{\ast}|+x_{1}^{\ast}=|x_{2}^{\ast}|+x_{2}^{\ast}.
\label{|x1+a|-(x1+a)=|x2+a|-(x2+a).}%
\end{equation}
Simularly, we use the same method for $w_{1}^{\ast}$ and $w_{2}^{\ast}$ we
have
\begin{equation}
|x_{1}^{\ast}|-x_{1}^{\ast}=|x_{2}^{\ast}|-x_{2}^{\ast}.
\label{|x1+a|+(x1+a)=|x2+a|+(x2+a).}%
\end{equation}
so $(\ref{|x1+a|-(x1+a)=|x2+a|-(x2+a).})-$($\ref{|x1+a|+(x1+a)=|x2+a|+(x2+a).}%
)$ means that $x_{1}^{\ast}=x_{2}^{\ast}$.
\end{remark}

Now, we give the following algorithm for solving $\tilde{F}(p,x)=0$ (see
\cite{Y. Shi}):

\textbf{Algorithm:}

\begin{enumerate}
\item[Step 0:] Determine $\epsilon$, $p$, $k^{\ast}$, $\rho,$ $\sigma$ such
that $k^{\ast}$ is a positive integer,

$0<\rho<1/2$, $\ $and $\rho<\sigma<1$;

\item[Step 1:] Select two points $x^{(0)\text{ }}$ and $x^{(1)}\in
\mathrm{IR}^{n}$;

\item[Step 2:] for $k=1,2,...$ until termination, do the following:

\begin{enumerate}
\item[1-] Compute the steepset descent direction
\[
d^{(k)}:=-J(p,x^{(k)})^{T}\tilde{F}(p,x^{(k)}),
\]
where
\[
\tilde{F}(p,x):=(M+I)x+\frac{1}{p}(M-I)\ln(e^{0}+e^{px}+e^{-px})+q;
\]%
\[
J(p,x):=(M+I)+(M-I)E(p,x)
\]
and
\[
E_{ij}(p,x):=\delta_{ij}\frac{e^{px_{i}}-e^{-px_{i}}}{1+e^{px_{i}}+e^{-px_{i}%
}};
\]
we racall that $\delta_{ii}=1$ and $\delta_{ij}=0$ if $i\neq j$.

\item[2-] if $k$ equals a multiple of $k^{\ast}$, then insert a steepset
descent direction step, that is, let $s^{(k)}:=d^{(k)}$ and go to step 2.7;

\item[3-] Compute:%
\[
\left\{
\begin{array}
[c]{l}%
u^{(k)}:=\xi_{1}\tilde{F}(p,x^{(k)})\\
v^{(k)}:=\xi_{2}(x^{(k)}-x^{(k-1)})
\end{array}
\right.
\]
with%
\[
\xi_{1}=-\frac{||x^{(k)}-x^{(k-1)}||^{2}}{<x^{(k)}-x^{(k-1)},\tilde
{F}(p,x^{(k)})-\tilde{F}(p,x^{(k-1)})>}%
\]
and%
\[
\xi_{2}=-\frac{<\tilde{F}(p,x^{(k)})-\tilde{F}(p,x^{(k-1)}),\tilde
{F}(p,x^{(k)})>}{||\tilde{F}(p,x^{(k)})-\tilde{F}(p,x^{(k-1)})||^{2}}.
\]

\item[4-] if $(u^{(k)}-v^{(k)})^{T}d^{(k)}\neq0$, then choose%
\[
\alpha^{(k)}>\frac{-<v^{(k)},d^{(k)}>}{<u^{(k)}-v^{(k)},d^{(k)}>}%
\]
such that $\alpha^{(k)}$ maximizes the value of
\[
Cos[s^{(k)},d^{(k)}]=\frac{<s^{(k)},d^{(k)}>}{||s^{(k)}||.||d^{(k)}||};
\]
Set
\[
s^{(k)}:=\alpha^{(k)}u^{(k)}+(1-\alpha^{(k)})v^{(k)}%
\]
and go to step 2.7;

\item[5-] if $(u^{(k)}-v^{(k)})^{T}d^{(k)}=0$ and $<v^{(k)},d^{(k)}>>0$, then
set
\[
s^{(k)}:=\frac{u^{(k)}+v^{(k)}}{2}%
\]
and go to step 2.7;

\item[6-] if $(u^{(k)}-v^{(k)})^{T}d^{(k)}=0$ and $<v^{(k)},d^{(k)}%
>\leqslant0$, then set $s^{(k)}:=d^{(k)}$ and go to step 2.7;

\item[7-] take a line search along the direction $s^{(k)}$ to determine the
step length $\gamma$ suth that%
\[
f(p,x^{(k)}+\gamma s^{(k)})\leqslant f(p,x^{(k)})-\gamma\rho<d^{(k)},s^{(k)}>
\]
and%
\[
<\bigtriangledown f(p,x^{(k)}+\gamma s^{(k)}),s^{(k)}>\geqslant-\sigma
<d^{(k)},s^{(k)}>;
\]
where%
\[
f(p,x)=\frac{1}{2}||\tilde{F}(p,x)||^{2}%
\]

\item[8-] set $x^{(k+1)}:=x^{(k)}+\gamma s^{(k)}$ and go to next iteration.
\end{enumerate}
\end{enumerate}

\item[\textbf{5.}] {\Large Numerical examples}

In this part, we consider some examples to test our method. The results of
Fixed Point and using vector divisions methods for these examples are
presented here for comparison purposes. The results and expected solutions for
each example have been presented on the Tables 1 and 2.

\begin{example}
: Consider the following linear complementarity problem:

Find vector $z$ in $\mathrm{IR}^{n}$ satisfying $z^{T}(Mz+q)=0$,
$Mz+q\geqslant0,$ $z\geqslant0$,

where $M=\left[
\begin{array}
[c]{cccc}%
4 & -1 & 0 & 0\\
-1 & 4 & -1 & 0\\
0 & -1 & 4 & -1\\
0 & 0 & -1 & 4
\end{array}
\right]  $ and $q=\left[
\begin{array}
[c]{c}%
-4\\
3\\
-4\\
2
\end{array}
\right]  $.

The exact solutions is $x^{\ast}=(1,0,1,0)^{T}.$

We apply the fixed point and using vector divisions methods to solve this
example with the initial approximation $x^{(0)}=(1.1,0.1,1.2,0.2)^{T}$.

The solution of this problem with six significant digits is presented in
Table1.
\[%
\begin{tabular}
[c]{llllll}\hline
& Iteration & $\ \ \ \ \ \ z_{1}$ & $\ \ \ \ \ \ z_{2}$ & $\ \ \ \ \ \ z_{3}$
& $\ \ \ \ \ \ z_{4}$\\\cline{2-6}
& k=01 & 0,0000000 & 0,0000000 & 0,0000000 & 0,0000000\\\cline{2-6}
& k=05 & 0,7883251 & 0,0000000 & 1,3448593 & 0,0000000\\\cline{2-6}
& k=10 & 1,0197946 & 0,0000000 & 0,9884737 & 0,0000000\\\cline{2-6}%
Fixed & k=15 & 0,9985643 & 0,0000000 & 1,0012907 & 0,0000000\\\cline{2-6}%
point & k=20 & 1,0001030 & 0,0000000 & 0,9999377 & 0,0000000\\\cline{2-6}%
method & k=25 & 0,9999918 & 0,0000000 & 1,0000058 & 0,0000000\\\cline{2-6}
& k=30 & 1,0000005 & 0,0000000 & 0,9999997 & 0,0000000\\\cline{2-6}
& k=33 & 0,9999999 & 0,0000000 & 1,0000001 & 0,0000000\\\cline{2-6}
& k=34 & 1,0000001 & 0,0000000 & 1,0000000 & 0,0000000\\\cline{2-6}
& k=35 & 1,0000000 & 0,0000000 & 1,0000000 & 0,0000000\\\hline
&  &  &  &  & \\\hline
Using vector & k=01 & 0,0000000 & 0,0000000 & 0,0000000 &
0,0000000\\\cline{2-6}%
divisions & k=02 & 4,0000000 & 0,0000000 & 4,0000000 & 0,0000000\\\cline{2-6}%
method & k=03 & 1,0000000 & 0,0000000 & 1,0000000 & 0,0000000\\\hline
\end{tabular}
\]
Table 1: The results of different methods for example1.
\end{example}

\begin{example}
: Let's solve the following linear complementarity problem

Find vector $z$ in $\mathrm{IR}^{n}$ satisfying $z^{T}(Mz+q)=0$,
$Mz+q\geqslant0,$ $z\geqslant0$,

where $M=\left[
\begin{array}
[c]{cccc}%
8 & -1 & 0 & -5\\
1 & 5 & -1 & 0\\
2 & -1 & 6 & -1\\
6 & 0 & -1 & 7
\end{array}
\right]  $ and $q=\left[
\begin{array}
[c]{c}%
1\\
-2\\
-3\\
4
\end{array}
\right]  $.

The exact solutions is $x^{\ast}=(0,\frac{15}{29},\frac{17}{29},0)^{T}.$

We apply the fixed point and using vector divisions methods to solve this
example with the initial approximation $x^{(0)}=(-1,-2,-3,-4)^{T}$.

The solution of this problem with six significant digits is presented in
Table2.%
\[%
\begin{tabular}
[c]{llllll}\hline
& Iteration & $\ \ \ \ \ \ z_{1}$ & $\ \ \ \ \ \ z_{2}$ & $\ \ \ \ \ \ z_{3}$
& $\ \ \ \ \ \ z_{4}$\\\cline{2-6}
& k=01 & 0,0000000 & 0,0000000 & 0,0000000 & 0,0000000\\\cline{2-6}
& k=05 & 0,1115059 & 0,6300237 & 0,3146258 & 0,0000000\\\cline{2-6}
& k=10 & 0,0000000 & 0,4906849 & 0,6581657 & 0,0000000\\\cline{2-6}
& k=15 & 0,0000000 & 0,5309163 & 0,5863067 & 0,0000000\\\cline{2-6}
& k=20 & 0,0000000 & 0,5192743 & 0,5936312 & 0,0000000\\\cline{2-6}
& k=25 & 0,0000000 & 0,5188339 & 0,5869623 & 0,0000000\\\cline{2-6}
& k=30 & 0,0000000 & 0,5173795 & 0,5862876 & 0,0000000\\\cline{2-6}%
Fixed & k=35 & 0,0000000 & 0,5171712 & 0,5859293 & 0,0000000\\\cline{2-6}%
point & k=40 & 0,0000000 & 0,5171411 & 0,5860432 & 0,0000000\\\cline{2-6}%
method & k=45 & 0,0000000 & 0,5171911 & 0,5861363 & 0,0000000\\\cline{2-6}
& k=50 & 0,0000000 & 0,5172239 & 0,5861902 & 0,0000000\\\cline{2-6}
& k=55 & 0,0000000 & 0,5172388 & 0,5862079 & 0,0000000\\\cline{2-6}
& k=60 & 0,0000000 & 0,5172428 & 0,5862107 & 0,0000000\\\cline{2-6}
& k=65 & 0,0000000 & 0,5172428 & 0,5860930 & 0,0000000\\\cline{2-6}
& k=70 & 0,0000000 & 0,5172421 & 0,5862079 & 0,0000000\\\cline{2-6}
& k=75 & 0,0000000 & 0,5172416 & 0,5862071 & 0,0000000\\\cline{2-6}
& k=79 & 0,0000000 & 0,5172424 & 0,5862069 & 0,0000000\\\hline
&  &  &  &  & \\\hline
& k=01 & 0,0000000 & 2,0000000 & 3,0000000 & 0,0000000\\\cline{2-6}
& k=02 & 0,0000000 & 0,0000000 & 0,4367816 & 0,0000000\\\cline{2-6}%
Using vector & k=03 & 0,0000000 & 2,6264367 & 0,5000000 &
0,0000000\\\cline{2-6}%
divisions & k=04 & 0,0000000 & 0,6983749 & 0,6163958 & 0,0000000\\\cline{2-6}%
method & k=05 & 0,0000000 & 0,5172414 & 0,5862069 & 0,0000000\\\cline{2-6}
& k=06 & 0,0000000 & 0,5172414 & 0,5862069 & 0,0000000\\\hline
\end{tabular}
\]
Table 2: The results of different methods for example2.
\end{example}
\end{enumerate}

\begin{center}
\textbf{Conclusion:}
\end{center}

In this paper we have used that the linear complementarity problem is
completely equivalent to solving nonlinear equation $F(x)=0$ with $F$ is a
function from $\mathrm{IR}^{n}$ into itself defined by
$F(x)=(M+I)x+(M-I)|x|+q$; for solving this equation we have used the Shi's
method, this method uses vector divisions with the secant method; based on
that, we have proposed a globally convergent hybrid algorithm for solving this
equation; to do so, we had to build a sequence of functions $\tilde{F}(p,x)\in
C^{\infty}$ which converges uniformly to the function $F(x)$; and we have
shown that finding the zero of the function $F$ is completely equivalent to
finding the zero of the sequence of the functions $\tilde{F}$.

\end{document}